\documentclass[titlepage,12pt]{article} 
\usepackage{hyperref}
\usepackage[usenames,dvipsnames]{pstricks} 
\usepackage{pst-plot} 
\usepackage[title]{appendix}
\usepackage{amssymb,amsthm,amsmath} 
\usepackage[a4paper]{geometry}
\usepackage{datetime2}
\usepackage[utf8]{inputenc}
\usepackage[italian,english]{babel}

\date{}

\def\NumeroVersione{4}

\date{Release \NumeroVersione\ (\DTMnow)}

\newcommand{\ep}{\varepsilon}
\newcommand{\re}{\mathbb{R}}

\newcommand{\n}{\mathbb{N}}

\newcommand{\holder}{H\"older}
\newcommand{\Hold}{\operatorname{Hold}}
\newcommand{\CC}{\mathcal{C}}
\newcommand{\GG}{\mathcal{G}}
\newcommand{\HH}{\mathcal{H}}
\newcommand{\XX}{\mathbb{X}}
\newcommand{\PS}{\mathcal{PS}}

\newtheorem{thm}{Theorem}[section]
\newtheorem{thmbibl}{Theorem}

\newtheorem{rmk}[thm]{Remark}
\newtheorem{prop}[thm]{Proposition}
\newtheorem{defn}[thm]{Definition}

\newtheorem{lemma}[thm]{Lemma}

\title{Critical counterexamples for linear wave equations with time-dependent propagation speed}

\author{Marina Ghisi\vspace{1ex}\\ 
{\normalsize Universit\`a degli Studi di Pisa} \\
{\normalsize Dipartimento di Matematica}\\ 
{\normalsize PISA (Italy)}\\
{\normalsize e-mail: \texttt{marina.ghisi@unipi.it}}
\and
Massimo Gobbino\vspace{1ex}\\ 
{\normalsize Universit\`a degli Studi di Pisa} \\
{\normalsize Dipartimento di Ingegneria Civile e Industriale}\\ 
{\normalsize PISA (Italy)}\\  
{\normalsize e-mail: \texttt{massimo.gobbino@unipi.it}}
}

\begin{document}
\maketitle

\begin{abstract}

We investigate an abstract wave equation with a time-dependent propagation speed, and we consider both the non-dissipative case, and the case with a strong damping that depends on a power of the elastic operator. Previous results show that, depending on the values of the parameters and on the time regularity of the propagation speed, this equation exhibits either well-posedness in Sobolev spaces, or well-posedness in Gevrey spaces, or ill-posedness with severe derivative loss.

In this paper we examine some critical cases that were left open by the previous literature, and we show that they fall into the pathological regime. The construction of the counterexamples requires a redesign from scratch of the basic ingredients, and a suitable application of Baire category theorem in place of the usual iteration scheme.

\vspace{6ex}

\noindent{\bf Mathematics Subject Classification 2010 (MSC2010):} 
35L90 (35L20, 35B30, 35B65).

		
\vspace{6ex}

\noindent{\bf Key words:} wave equation, propagation speed, strong damping, derivative loss, \holder\ continuity, Gevrey spaces, ultradistributions, Baire category, residual set.

\end{abstract}

 
\section{Introduction}

Let $\HH$ be a Hilbert space, and let $A$ be a linear nonnegative self-adjoint operator on $\HH$. We consider the evolution equation
\begin{equation}
u''(t)+2\delta A^{\sigma}u'(t)+c(t)Au(t)=0,
\label{eqn:main}
\end{equation}
with initial data
\begin{equation}
u(0)=u_{0},
\qquad
u'(0)=u_{1}.
\label{eqn:main-data}
\end{equation}

Here $\delta\geq 0$ and $\sigma\geq 0$ are real numbers, and $c:[0,+\infty)\to[0,+\infty)$ is a given function that we call ``propagation speed'' in analogy with the wave equation. We always assume that the propagation speed satisfies the strict hyperbolicity condition
\begin{equation}
0<\mu_{1}\leq c(t)\leq\mu_{2}
\qquad
\forall t\geq 0,
\label{hp:c-sh}
\end{equation}
and the \holder\ continuity condition
\begin{equation}
|c(t)-c(s)|\leq H|t-s|^{\alpha}
\qquad
\forall(t,s)\in[0,+\infty)^{2},
\label{hp:c-holder}
\end{equation}
for suitable real constants $\mu_{1}$, $\mu_{2}$, $H$, and $\alpha\in(0,1)$.
Many papers have been devoted to equations of this type (see for example the classical references~\cite{dgcs,cjs} or the more recent ones~\cite{2006-JDE-CicCol,2007-JMAA-CicCol,2008-MScand-CicHir,2002-SNS-ColDSaKin,2003-BSM-ColDSaRei,2013-AdNS-DAbLuc,2011-JDE-DAbRei,2015-JMAA-EbeFitHir,gg:dgcs-strong,gg:cjs-strong,2004-SNS-HirRei,2009-JMAA-HirWir}). Let us briefly discuss the previous results that are more relevant to our presentation.

\paragraph{\textmd{\textit{The non-dissipative case}}}

The case $\delta=0$ was addressed in the seminal paper~\cite{dgcs}. The general philosophy is that higher space regularity of initial data compensates lower regularity of the propagation speed. We refer to section~\ref{sec:previous} for precise definitions and statements, but the situation can be roughly described as follows (see also the figures in section~\ref{sec:statements}).
\begin{itemize}

\item If $c(t)$ is Lipschitz continuous, or more generally has locally bounded variation, then  problem (\ref{eqn:main})--(\ref{eqn:main-data}) is well-posed in $D(A^{1/2})\times H$, or more generally in Sobolev spaces of the form $D(A^{\beta+1/2})\times D(A^{\beta})$.

\item  If $c(t)$ is \holder\ continuous of order $\alpha$, then problem (\ref{eqn:main})--(\ref{eqn:main-data}) is
\begin{itemize}

\item  \emph{globally} well-posed in Gevrey spaces of order $s<(1-\alpha)^{-1}$,

\item  \emph{locally} well-posed in Gevrey spaces of order $s=(1-\alpha)^{-1}$,

\item  ill-posed in Gevrey spaces of order $s>(1-\alpha)^{-1}$. More precisely, there exist a propagation speed $c(t)$ that is \holder\ continuous of order $\alpha$, and a pair of initial conditions $(u_{0},u_{1})$ that are in the Gevrey class of order $s$ for every $s>(1-\alpha)^{-1}$, such that the corresponding solution to (\ref{eqn:main})--(\ref{eqn:main-data}) (which always exists in a very weak sense) is not even a distribution for all positive times. We call (DGCS)-phenomenon this instantaneous severe derivative loss.

\end{itemize}

\end{itemize}

We stress that in the critical case $s=(1-\alpha)^{-1}$ the well-posedness result of~\cite{dgcs} is just local-in-time, meaning that the solution is guaranteed to remain regular only in a finite time interval $[0,t_{0}]$. 

A first result of this paper addresses this critical case. Indeed, we show an example (and actually a residual set of examples) where the solution exhibits the (DGCS)-phenomenon after a finite time interval, thus proving the optimality of the local result in~\cite{dgcs}.

\paragraph{\textmd{\textit{The dissipative case}}}

The case $\delta>0$ was addressed in~\cite{gg:dgcs-strong}. The general philosophy is that there is a competition between the strong damping and the potential low regularity of the propagation speed. If $\sigma\geq 1/2$ the dissipation always wins, even if $c(t)$ is just continuous (independently of the continuity modulus). In this regime the equation behaves as in the case of a constant propagation speed, which means well-posedness in several classes of Sobolev spaces (see~\cite{ggh:tams,gg:dgcs-strong}). So the competition is more interesting when $\sigma<1/2$, where we have three possibilities.

\begin{itemize}

\item  If $c(t)$ is \holder\ continuous of order $\alpha>1-2\sigma$, then the dissipation prevails, and again the problem behaves as in the case of a constant propagation speed, namely it is well-posed in Sobolev spaces such as $D(A^{1/2})\times H$ or $D(A^{\beta+1/2})\times D(A^{\beta})$.

\item  If $c(t)$ is \holder\ continuous of order $\alpha<1-2\sigma$, then the dissipation can be neglected, and the behavior is the same as in the non-dissipative case, meaning
\begin{itemize}

\item  \emph{global} well-posedness in Gevrey spaces of order $s<(1-\alpha)^{-1}$,

\item  \emph{local} well-posedness in Gevrey spaces of order $s=(1-\alpha)^{-1}$,

\item  possibility of (DGCS)-phenomenon in Gevrey spaces of order $s>(1-\alpha)^{-1}$.
\end{itemize}

\item  If $c(t)$ is \holder\ continuous of order $\alpha=1-2\sigma$, and $\delta$ is large enough, then again the damping prevails, and one obtains well-posedness in Sobolev spaces.

\end{itemize}

Two cases were left open.
\begin{enumerate}
\renewcommand{\labelenumi}{(\arabic{enumi})}

\item  The case where $\alpha<1-2\sigma$ and $s=(1-\alpha)^{-1}$. In this paper (see Theorem~\ref{thm:gevrey-critical}) we show that solutions can exhibit the (DGCS)-phenomenon after a finite time, meaning that the local well-posedness result is optimal and can not be improved to global well-posedness. Our examples cover also to the non-dissipative case $\delta=0$.

\item  The case where $\alpha=1-2\sigma$ and $\delta$ is small enough. Also in this case we show (see Theorem~\ref{thm:damping-critical}) that the (DGCS)-phenomenon is possible, exactly as in the non-dissipative case.
 
\end{enumerate}

\paragraph{\textmd{\textit{Overview of the technique}}}

From the technical point of view, the spectral theorem reduces the problem to estimating the growth of solutions to the family of ordinary differential equations
\begin{equation}
u_{\lambda}''(t)+2\delta\lambda^{2\sigma}u_{\lambda}'(t)+\lambda^{2}c(t)u_{\lambda}(t)=0,
\label{eqn:u-lambda}
\end{equation}
with initial data
\begin{equation}
u_{\lambda}(0)=u_{0,\lambda},
\qquad
u_{\lambda}'(0)=u_{1,\lambda}.
\nonumber
\end{equation}

Now we give a brief heuristic presentation of the main ideas behind the previous and present results. For the sake of clarity, at the risk of cheating a little bit from time to time, we do not quote the exact estimates with all technical details, for which the interested reader is referred to the original papers. 

Let us consider the usual energy
\begin{equation}
E_{\lambda}(t):=|u_{\lambda}'(t)|^{2}+\lambda^{2}|u_{\lambda}(t)|^{2}.
\nonumber
\end{equation}

If $c(t)$ is \holder\ continuous of order $\alpha$, the approximated energy estimates introduced in~\cite{dgcs}, and then extended in~\cite{gg:dgcs-strong} to the dissipative case, yield inequalities of the form
\begin{equation}
E_{\lambda}(t)\leq c_{1 }E_{\lambda}(0)\exp\left(c_{2}\lambda^{1-\alpha}t-c_{3}\delta\lambda^{2\sigma}t\right)
\qquad
\forall t\geq 0,
\nonumber
\end{equation}
where $c_{1}$, $c_{2}$, $c_{3}$ are positive constants that depend on $\mu_{1}$, $\mu_{2}$, and on the \holder\ constant of $c(t)$, but are independent of $\delta$ and $\lambda$. Estimates of this kind are the core of all the well-posedness results quoted above. They also explain the competition between $1-\alpha$ and $2\sigma$, and why the size of $\delta$ becomes relevant if and only if $1-\alpha=2\sigma$.

On the contrary, the (DGCS)-phenomenon originates from estimates on the opposite side. More precisely, it was shown (in~\cite{dgcs} in the case $\delta=0$, and in~\cite{gg:dgcs-strong} when $\delta>0$) that for every $\lambda$ there exists a propagation speed $c_{\lambda}(t)$, \holder\ continuous of order $\alpha$ with a constant that does not depend on $\lambda$, such that equation (\ref{eqn:u-lambda}) with $c(t):=c_{\lambda}(t)$ admits a nontrivial solution that satisfies
\begin{equation}
E_{\lambda}(t)\geq c_{4}E_{\lambda}(0)\exp\left(c_{5}\lambda^{1-\alpha}t-c_{6}\delta\lambda^{2\sigma}t\right)
\qquad
\forall t\geq 0,
\nonumber
\end{equation}
where again the constants $c_{4}$, $c_{5}$, $c_{6}$ are positive and do not depend on $\delta$ and $\lambda$. In this case we say that $c_{\lambda}(t)$ ``activates'' the frequency $\lambda$. Roughly speaking, this is possible because of a resonance effect between the oscillations of $c_{\lambda}(t)$, and the ``natural'' oscillations of solutions to the same equation with constant propagation speed. The big problem is that in this construction $c_{\lambda}(t)$ \emph{does depend} on $\lambda$.

In order to overcome this difficulty, a very clever iterative procedure was devised in~\cite{dgcs}, and then exploited so far in the literature. In a nutshell, one chooses a sequence of frequencies $\{\lambda_{k}\}$ that grows fast enough, and a decreasing sequence $\{t_{k}\}$ of positive times that goes to 0 fast enough. Then one defines a propagation speed $c(t)$ that coincides in $[t_{k+1},t_{k}]$ with the propagation speed $c_{\lambda_{k}}(t)$ that activates the frequency $\lambda_{k}$. If all the parameter are chosen in a clever way, the resulting propagation speed is \holder\ continuous of order $\alpha$, and for every positive integer $k$ equation (\ref{eqn:u-lambda}) with $\lambda:=\lambda_{k}$ admits a nontrivial solution such that
\begin{equation}
E_{\lambda_{k}}(t)\geq c_{7}E_{\lambda_{k}}(0)\exp\left(c_{8}\frac{\lambda_{k}^{1-\alpha}}{\log(1+\lambda_{k})}t-c_{9}\delta\lambda_{k}^{2\sigma}t\right)
\qquad
\forall t\geq 0,
\nonumber
\end{equation}
where again the constants do not depend on $\delta$ and $\lambda_{k}$.

In other words, now $c(t)$ does not depend on the frequency $\lambda_{k}$, but we had to pay a little in the growth rate (actually the $\log(1+\lambda_{k})$ can be replaced by any given unbounded function). The payment comes from the fact that $c(t)$ has to activate infinitely many frequencies, and it activates them \emph{one-by-one} in time intervals of (necessarily) \emph{vanishing lengths}.

This construction opens the door to all the instances of the (DGCS)-phenomenon that we mentioned above. On the other hand, it is clear that it can \emph{not} help for critical values of the parameters. In order to address these cases, we need a propagation speed $c(t)$, independent of $\lambda$, such that equation (\ref{eqn:u-lambda}) with $\lambda:=\lambda_{k}$ admits a nontrivial solution such that
\begin{equation}
E_{\lambda_{k}}(t)\geq c_{10}E_{\lambda_{k}}(0)\exp\left(c_{11}\lambda_{k}^{1-\alpha}t-c_{12}\delta\lambda_{k}^{2\sigma}t\right)
\qquad
\forall t\geq 0.
\nonumber
\end{equation}

This is the main technical contribution of this paper, namely a propagation speed $c(t)$ that \emph{activates an unbounded set of frequencies in the same time}. 

Our construction has two main steps.
\begin{itemize}

\item  In the first step (see section~\ref{sec:basic}) we show that any propagation speed $c_{0}(t)$ that is smooth enough can be modified in order to obtain a propagation speed $c_{\lambda}(t)$ that activates a large enough frequency $\lambda$. We can also assume that $c_{\lambda}-c_{0}$ is as small as we want in the uniform norm, and that the \holder\ constant of $c_{\lambda}$ is as close as we want to the \holder\ constant of $c_{0}$. Note that the \holder\ constant of $c_{\lambda}-c_{0}$ is not necessarily small, and actually it is of the same order as the \holder\ constant of $c_{0}$. In other words, in this example the \holder\ constant of the sum of two functions is the maximum, and not the sum, of the \holder\ constants of the two terms. This sounds somewhat counterintuitive, and it is possible because the two terms ``oscillate at different frequencies'' (see Lemma~\ref{lemma:holder}).

\item  In the second step (see section~\ref{sec:baire}) we apply Baire category theorem in order to show that the set of propagation speeds that activate countably many frequences in the same time is residual in the set of all admissible propagation speeds. In this way we avoid the technicalities of the iteration scheme, and we leave all the dirty work to the abstract result.

\end{itemize}

The conclusion is not the construction of a single counterexample, resulting from a sum of lucky circumstances and clever choices, but a proof that \emph{the (DGCS)-phenomenon is the typical behavior} when the assumptions of the classical well-posedness results are not satisfied. In~\cite{gg:residual} we observed the same issue in the non-critical cases, and in different examples from geometric measure theory and transport equations.	

In this paper we focussed on \holder\ continuous propagation speeds, but we are confident that these techniques could be useful in the construction of counterexamples also when the propagation speed satisfies different regularity conditions (see for example \cite{2006-JDE-CicCol,2007-JMAA-CicCol,2003-BSM-ColDSaRei,2004-SNS-HirRei}).

\paragraph{\textmd{\textit{Structure of the paper}}}

In section~\ref{sec:previous} we introduce the functional setting, and we review the previous results that are relevant to this paper. In section~\ref{sec:statements} we state our two main contributions, and we show how they complete the picture of regularity results for solutions to (\ref{eqn:main}). In section~\ref{sec:proofs} we show the existence of a residual set of ``universal activators'', namely admissible propagation speeds that activate countably many frequencies, and we use these propagation speeds in order to prove our main results. Finally, in the appendix we present a heuristic motivation of the (otherwise somewhat mysterious) construction that we made.


\setcounter{equation}{0}
\section{Notations and previous work}\label{sec:previous}

\paragraph{\textmd{\textit{Functional spaces}}}

Let $\HH$ be a Hilbert space, and let $A$ be a linear nonnegative self-adjoint operator on $\HH$. Just for simplicity, we always make the following assumption.

\begin{defn}[Nonnegative discrete multiplication operator]
\begin{em}

Let $A$ be a linear continuous operator on a Hilbert space $\HH$. We say that $A$ is a \emph{nonnegative discrete multiplication operator} if there exist an orthonormal basis $\{e_{i}\}$ of $\HH$, and a nondecreasing sequence $\{\lambda_{i}\}$ of nonnegative real numbers such that 
\begin{equation}
Ae_{i}=\lambda_{i}^{2}e_{i}
\qquad
\forall i\in\n.
\nonumber
\end{equation}

In addition, we say that the operator is \emph{unbounded} if $\lambda_{i}\to+\infty$ as $i\to +\infty$.
\end{em}
\end{defn}

As explained in~\cite{gg:dgcs-strong,gg:cjs-strong}, there is almost no loss of generality in this assumption, because the spectral theorem in its general form states that every self-adjoint continuous operator in a Hilbert space behaves as a multiplication operator in some $L^{2}$ space with respect to some (not necessarily discrete) metric space.

Thanks to the orthonormal basis, we can identify every element $u\in\HH$ with the sequence $\{u_{i}\}\in\ell^{2}$ of its ``Fourier'' components. This identification can be extended in order to define Sobolev spaces, Gevrey spaces and (hyper)distributions. Several choices are possible (see~\cite{gg:dgcs-strong,gg:cjs-strong}). Here we recall the definitions that are needed in the sequel.

\begin{defn}
\begin{em}

Let $u$ be a sequence $\{u_{i}\}$ of real numbers.
\begin{itemize}

\item  \emph{Sobolev spaces and distributions}. Let $\beta$ be a real number. We say that $u\in D(A^{\beta})$ if
\begin{equation}
\|u\|_{D(A^{\beta})}^{2}:=\sum_{i=0}^{\infty}u_{i}^{2}(1+\lambda_{i})^{4\beta}<+\infty.
\nonumber
\end{equation}

\item  \emph{Gevrey spaces}. Let $\beta$ be a real number, and let $s$ and $r$ be positive real numbers. We say that $u\in\GG_{s,r,\beta}(A)$ if
\begin{equation}
\|u\|_{\GG_{s,r,\beta}(A)}^{2}:=\sum_{i=0}^{\infty}u_{i}^{2}(1+\lambda_{i})^{4\beta}\exp\left(2r\lambda_{i}^{1/s}\right)<+\infty.
\nonumber
\end{equation}

\item  \emph{Gevrey ultradistributions}. Let $\beta$ be a real number, and let $S$ and $R$ be positive real numbers. We say that $u\in\GG_{-S,R,\beta}(A)$ if
\begin{equation}
\|u\|_{\GG_{-S,R,\beta}(A)}^{2}:=\sum_{i=0}^{\infty}u_{i}^{2}(1+\lambda_{i})^{4\beta}\exp\left(-2R\lambda_{i}^{1/S}\right)<+\infty.
\nonumber
\end{equation}

\end{itemize}

\end{em}
\end{defn}

We refer to~\cite[Remark~2.2 and Remark~2.3]{gg:dgcs-strong} for further comments on these spaces, and for consistency with the classical setting. In the case of Gevrey spaces and ultradistributions, we refer to $s$ and $S$ as ``the order'', and we refer to $r$ and $R$ as ``the radius''. The parameter $\beta$ represents some sort of further Sobolev regularity.

\paragraph{\textmd{\textit{Admissible propagation speeds}}}

In this paper we restrict to the following class of admissible propagation speeds. 

\begin{defn}[Admissible propagation speeds]
\begin{em}

Let $\mu_{1}$, $\mu_{2}$, $\alpha$, $H$ be real numbers such that
\begin{equation}
0<\mu_{1}<\mu_{2},
\qquad\quad
\alpha\in(0,1),
\qquad\quad
H>0.
\label{hp:accH}
\end{equation}

The set of admissible propagations speeds is the set $\PS(\mu_{1},\mu_{2},\alpha,H)$ of all functions $c:[0,+\infty)\to[0,+\infty)$ satisfying the strict hyperbolicity condition (\ref{hp:c-sh}) and the \holder\ continuity condition (\ref{hp:c-holder}).

\end{em}
\end{defn}

We observe that $\PS(\mu_{1},\mu_{2},\alpha,H)$ is a complete metric space with respect to the distance induced by the uniform norm.

\paragraph{\textmd{\textit{Very weak solutions}}}

When $A$ is a nonnegative multiplication operator, one can reduce problem  (\ref{eqn:main})--(\ref{eqn:main-data}) to the following (uncoupled) infinite system of ordinary differential equations
\begin{equation}
u_{i}''(t)+2\delta\lambda_{i}^{2\sigma}u_{i}'(t)+c(t)\lambda_{i}^{2}u_{i}(t)=0
\qquad
\forall i\in\n,\quad\forall t\geq 0,
\nonumber
\end{equation}
with initial data
\begin{equation}
u_{i}(0)=u_{0,i},
\qquad
u_{i}'(0)=u_{1,i}.
\nonumber
\end{equation}

A \emph{very weak solution} to (\ref{eqn:main})--(\ref{eqn:main-data}) is a sequence $\{u_{i}(t)\}$ of solutions to this system, and for trivial reasons it exists and is unique for every choice of the sequences of initial data $\{u_{0,i}\}$ and $\{u_{1,i}\}$. The main point is understanding the regularity of the sequence $\{u_{i}(t)\}$ in terms of the regularity of initial data.

\paragraph{\textmd{\textit{Previous results}}}

The first result concerns ``regularity'' in a huge space of hyperdistributions, both for the dissipative and for the non-dissipative equation. From the point of view of counterexamples, this represents some sort of ``bound from below'', namely a minimal regularity that cannot be lost during the evolution.

\begin{thmbibl}[Well-posedness in Gevrey hyperdistributions, see~{\cite[Theorem~3]{dgcs}}]\label{thmbibl:hyper}

Let us consider problem (\ref{eqn:main})--(\ref{eqn:main-data}) under the following assumptions:
\begin{itemize}

\item $A$ is a nonnegative discrete multiplication operator on a Hilbert space $\HH$,

\item  $c\in\PS(\mu_{1},\mu_{2},\alpha,H)$ for suitable values of the parameters satisfying (\ref{hp:accH}),

\item  $\delta\geq 0$ and $\sigma\geq 0$ are real numbers,

\item  there exist real numbers $R_{0}>0$ and $0<S\leq(1-\alpha)^{-1}$ such that
\begin{equation}
(u_{0},u_{1})\in\GG_{-S,R_{0},1/2}(A)\times\GG_{-S,R_{0},0}(A).
\nonumber
\end{equation}

\end{itemize}

Then there exists $R>0$ such that the unique solution $u$ satisfies
\begin{equation}
u\in C^{0}\left([0,+\infty),\GG_{-S,R_{0}+Rt,1/2}(A)\right)\cap C^{1}\left([0,+\infty),\GG_{-S,R_{0}+Rt,0}(A)\strut\right).
\label{th:scale-hyper}
\end{equation}

\end{thmbibl}

Condition (\ref{th:scale-hyper}), with the range space increasing with time, simply means that 
$$u\in C^{0}\left([0,\tau],\GG_{-S,R_{0}+R\tau,1/2}(A)\right)\cap C^{1}\left([0,\tau],\GG_{-S,R_{0}+R\tau,0}(A)\strut\right)
\quad\quad
\forall\tau> 0.$$
		
The second result concerns well-posedness in Gevrey spaces of suitable order, both for the dissipative and for the non-dissipative equation.

\begin{thmbibl}[Well-posedness in Gevrey spaces, see~{\cite[Theorem~2]{dgcs}}]\label{thmbibl:gevrey}

Let us consider problem (\ref{eqn:main})--(\ref{eqn:main-data}) under the following assumptions:
\begin{itemize}

\item $A$ is a nonnegative discrete multiplication operator on a Hilbert space $\HH$,

\item  $c\in\PS(\mu_{1},\mu_{2},\alpha,H)$ for suitable values of the parameters satisfying (\ref{hp:accH}),

\item  $\delta\geq 0$ and $\sigma\geq 0$ are real numbers,

\item  there exist real numbers $r_{0}>0$ and $0<s\leq(1-\alpha)^{-1}$ such that
\begin{equation}
(u_{0},u_{1})\in\GG_{s,r_{0},1/2}(A)\times\GG_{s,r_{0},0}(A).
\nonumber
\end{equation}

\end{itemize}

Then the following statements hold true.
\begin{enumerate}
\renewcommand{\labelenumi}{(\arabic{enumi})}

\item  \emph{(Global-in-time regularity)} In the case $0<s<(1-\alpha)^{-1}$, the unique solution $u$ to the problem satisfies 
\begin{equation}
u\in C^{0}\left([0,T],\GG_{s,r_{1},1/2}(A)\right)\cap C^{1}\left([0,T],\GG_{s,r_{1},0}(A)\strut\right)
\qquad
\forall T>0,
\quad
\forall r_{1}<r_{0}.
\nonumber
\end{equation}

\item  \emph{(Local-in-time regularity)} In the case $s=(1-\alpha)^{-1}$, there exist real numbers $t_{0}>0$ and $r>0$, with $r$ independent of $r_{0}$, such that $rt_{0}<r_{0}$ and the unique solution $u$ to the problem satisfies
\begin{equation}
u\in C^{0}\left([0,t_{0}],\GG_{s,r_{0}-rt,1/2}(A)\right)\cap C^{1}\left([0,t_{0}],\GG_{s,r_{0}-rt,0}(A)\strut\right).
\label{th:scale-gevrey}
\end{equation}

\end{enumerate}

\end{thmbibl}

We observe that also in Theorem~\ref{thmbibl:gevrey} above the range space in (\ref{th:scale-gevrey}) is increasing with time. More important, in the critical case $s=(1-\alpha)^{-1}$, this is a local well-posedness result because after a finite time the solution might have lost all its initial radius $r_{0}$ (and this actually happens, as we are going to show in this paper). 

In the next result the strong damping comes into play for the first time, providing well-posedness in Sobolev spaces if the propagation speed in ``enough \holder\ continuous''.

\begin{thmbibl}[Well-posedness in Sobolev spaces, see~{\cite[Theorem~3.2]{gg:dgcs-strong}}]\label{thmbibl:sobolev}

Let us consider problem (\ref{eqn:main})--(\ref{eqn:main-data}) under the following assumptions:
\begin{itemize}

\item $A$ is a nonnegative discrete multiplication operator on a Hilbert space $\HH$,

\item  $c\in\PS(\mu_{1},\mu_{2},\alpha,H)$ for suitable values of the parameters satisfying (\ref{hp:accH}),

\item  $\delta$ and $\sigma$ are real numbers such that either $2\sigma>1-\alpha$ and $\delta>0$, or $2\sigma=1-\alpha$ and $\delta$ is large enough.

\item  $(u_{0},u_{1})\in D(A^{1/2})\times H$.

\end{itemize}

Then the unique solution $u$ to the problem satisfies
\begin{equation}
u\in C^{0}\left([0,+\infty),D(A^{1/2})\right)\cap C^{1}\left([0,+\infty),H\strut\right).
\nonumber
\end{equation}

\end{thmbibl}

Finally, the last result is the counterpart of Theorem~\ref{thmbibl:sobolev}. It shows that the (DGCS)-phenomenon can happen, despite the strong damping, if the propagation speed is not ``enough \holder\ continuous''. We point out that the derivative loss is as severe as allowed by Theorem~\ref{thmbibl:hyper}.

\begin{thmbibl}[Severe derivative loss, see~\cite{gg:dgcs-strong}]\label{thmbibl:dgcs}

Let $\HH$ be a Hilbert space, and let $A$ be a nonnegative discrete multiplication operator that we assume to be unbounded. Let $\mu_{1}$, $\mu_{2}$, $\alpha$, $H$ be real numbers satisfying (\ref{hp:accH}). Let $\delta$ and $\sigma$ be real numbers such that
\begin{equation}
\delta\geq 0
\qquad\mbox\qquad
2\sigma<1-\alpha.
\nonumber
\end{equation}

Then there exist a propagation speed $c\in\PS(\mu_{1},\mu_{2},\alpha,H)$, and a very weak solution $u$ to equation (\ref{eqn:main}), such that
\begin{equation}
(u(0),u'(0))\in\GG_{s,r,1/2}(A)\times\GG_{s,r,0}(A)
\qquad
\forall s>\frac{1}{1-\alpha},
\quad
\forall r>0,
\nonumber
\end{equation}
but
\begin{equation}
(u(t),u'(t))\not\in\GG_{-S,R,1/2}(A)\times\GG_{-S,R,0}(A)
\qquad
\forall S>\frac{1}{1-\alpha},
\quad
\forall R>0,
\quad
\forall t>0.
\nonumber
\end{equation}

\end{thmbibl}


\setcounter{equation}{0}
\section{Statement of our results}\label{sec:statements}

In order to state our results in a more compact way, we introduce two variants of Gevrey spaces and hyperdistributions.

\begin{defn}
\begin{em}

Let $u$ be a sequence $\{u_{i}\}$ of real numbers.
\begin{itemize}

\item  Let $\beta$ and $s$ be real numbers, with $s>0$. We say that $u\in\GG_{s,\log,\beta}(A)$ if
\begin{equation}
\|u\|_{\GG_{s,\log,\beta}(A)}^{2}:=\sum_{i=0}^{\infty}u_{i}^{2}(1+\lambda_{i})^{4\beta}\exp\left(\frac{2\lambda_{i}^{1/s}}{\log(2+\lambda_{i})}\right)<+\infty.
\nonumber
\end{equation}

\item  Let $\beta$ and $S$ be real numbers, with $S>0$. We say that $u\in\GG_{-S,\log,\beta}(A)$ if
\begin{equation}
\|u\|_{\GG_{-S,\log,\beta}(A)}^{2}:=\sum_{i=0}^{\infty}u_{i}^{2}(1+\lambda_{i})^{4\beta}\exp\left(-\frac{2\lambda_{i}^{1/S}}{\log(2+\lambda_{i})}\right)<+\infty.
\nonumber
\end{equation}

\end{itemize}

\end{em}
\end{defn}

The key property of these spaces are the following two implications:
\begin{eqnarray}
u\in\GG_{s,\log,\beta}(A) & \Longrightarrow & \forall s'>s,\ \forall r>0,\ \forall\gamma\in\re\quad u\in\GG_{s',r,\gamma}(A),
\label{log-gevrey}  \\[1ex]
u\not\in\GG_{-S,\log,\beta}(A) & \Longrightarrow & \forall S'>S,\ \forall R>0,\ \forall\gamma\in\re\quad u\not\in\GG_{-S',R,\gamma}(A).
\label{log-hyper}
\end{eqnarray}

We note that the same properties hold true if the logarithm is replaced by any function that tends to $+\infty$ as $\lambda_{i}\to +\infty$.


Our first result concerns the local nature of Theorem~\ref{thmbibl:gevrey}, both in the non-dissipative and in the dissipative case. We show that an initial condition, with finite radius in a Gevrey space of critical order, can undergo, during the evolution, a degradation of its radius and become a hyperdistribution (and nothing more) after a finite time.

\begin{thm}[Severe derivative loss for large times for critical Gevrey index]\label{thm:gevrey-critical}

Let $\HH$ be a Hilbert space, and let $A$ be a nonnegative discrete multiplication operator that we assume to be unbounded. Let $\mu_{1}$, $\mu_{2}$, $\alpha$, $H$ be real numbers satisfying (\ref{hp:accH}). Let $\delta$, $\sigma$, $s$, $S$, $r_{0}$ be real numbers such that
\begin{equation}
\delta\geq 0,
\qquad\quad
2\sigma<1-\alpha,
\qquad\quad
s=S=\frac{1}{1-\alpha},
\qquad\quad
r_{0}>0.
\nonumber
\end{equation}

Let us set
\begin{equation}
t_{0}:=\frac{32\mu_{2}^{(1+\alpha)/2}}{H}\cdot r_{0}.
\nonumber
\end{equation}

Then the set of propagation speeds $c\in\PS(\mu_{1},\mu_{2},\alpha,H)$ for which equation (\ref{eqn:main}) admits a solution satisfying
\begin{equation}
(u(0),u'(0))\in\GG_{s,r_{0},1/2}(A)\times\GG_{s,r_{0},0}(A)
\nonumber
\end{equation}
and 
\begin{equation}
(u(t),u'(t))\not\in\GG_{-S,\log,1/2}(A)\times\GG_{-S,\log,0}(A)
\qquad
\forall t>t_{0}
\label{th:dgcs-gevrey}
\end{equation}
is residual in $\PS(\mu_{1},\mu_{2},\alpha,H)$ (with respect to the $L^{\infty}$ distance).

\end{thm}

We point out that, due to (\ref{log-hyper}), a derivative loss of the form (\ref{th:dgcs-gevrey}) implies a derivative loss of the form
\begin{equation}
(u(t),u'(t))\not\in\GG_{-S',R,1/2}(A)\times\GG_{-S',R,0}(A)
\qquad
\forall S'>\frac{1}{1-\alpha},
\quad
\forall R>0,
\quad
\forall t> t_{0},
\nonumber
\end{equation}
which is the largest possible derivative loss compatible with Theorem~\ref{thmbibl:hyper}.


Our second result concerns the dissipative equation, and addresses the critical case where $\alpha=1-2\sigma$ and $\delta$ is small enough. In this regime, if initial data are ``not enough Gevrey regular'', solutions can undergo an instantaneous derivative loss, as severe as allowed by Theorem~\ref{thmbibl:hyper}. We recall that, with the same values of $\alpha$ and $\sigma$, but large enough $\delta$, Theorem~\ref{thmbibl:sobolev} shows well-posedness in Sobolev spaces.

\begin{thm}[Instantaneous severe derivative loss for small critical damping]\label{thm:damping-critical}

Let $\HH$ be a Hilbert space, and let $A$ be a nonnegative discrete multiplication operator that we assume to be unbounded. Let $\mu_{1}$, $\mu_{2}$, $\alpha$, $H$ be real numbers satisfying (\ref{hp:accH}). Let $\delta$, $\sigma$, $s$, $S$ be real numbers such that
\begin{equation}
0\leq\delta<\frac{H}{32\mu_{2}^{(1+\alpha)/2}},
\qquad\quad
2\sigma=1-\alpha,
\qquad\quad
s=S=\frac{1}{1-\alpha}.
\nonumber
\end{equation}

Then the set of propagation speeds $c\in\PS(\mu_{1},\mu_{2},\alpha,H)$ for which equation (\ref{eqn:main}) admits a solution satisfying
\begin{equation}
(u(0),u'(0))\in\GG_{s,\log,1/2}(A)\times\GG_{s,\log,0}(A)
\label{hp:dgcs-gevreydata}
\end{equation}
and 
\begin{equation}
(u(t),u'(t))\not\in\GG_{-S,\log,1/2}(A)\times\GG_{-S,\log,0}(A)
\qquad
\forall t> 0
\label{th:dgcs-gevrey-bis}
\end{equation}
is residual in $\PS(\mu_{1},\mu_{2},\alpha,H)$  (with respect to the $L^{\infty}$ distance).

\end{thm}

We point out that, due to (\ref{log-gevrey}), condition (\ref{hp:dgcs-gevreydata}) implies that initial data are as close as possible to the Gevrey space that would guarantee regularity of solutions according to Theorem~\ref{thmbibl:gevrey}.

\begin{rmk}\label{rmk:data-residual}
\begin{em}

When (\ref{eqn:main}) has a solution with some derivative loss, then there is actually a residual set of solutions with the same derivative loss. More precisely, let us assume that, for some propagation speed $c(t)$, equation (\ref{eqn:main}) admits a solution satisfying (\ref{hp:dgcs-gevreydata}) and (\ref{th:dgcs-gevrey-bis}).  Then the set of initial data for which the solutions satisfies (\ref{th:dgcs-gevrey-bis}) is residual in the space that appears in (\ref{hp:dgcs-gevreydata}). This is again an application of the Baire category theorem (see section~\ref{sec:rmk}). An analogous remark applies to Theorem~\ref{thm:gevrey-critical}.

\end{em}
\end{rmk}

\begin{rmk}
\begin{em}

In the statements of all our results concerning derivative loss, we have always assumed that $\delta\geq 0$. This is just because we are focussing on equations either without dissipation, or with a ``true'' dissipation. On the other hand, those results hold true a fortiori if $\delta<0$, namely when the ``dissipation'' has the wrong sign.

\end{em}
\end{rmk}


The results of this paper should cover all the cases that were left open in previous literature, at least in the strictly hyperbolic case with \holder\ continuous propagation speed. The following pictures summarize the final state of the art. In the horizontal axis we represent the time-regularity of $c(t)$. With some abuse of notation, values $\alpha\in(0,1)$ mean that $c(t)$ is $\alpha$-H\"{o}lder continuous, $\alpha=1$ means that it is Lipschitz continuous, $\alpha>1$ means further regularity.  In the vertical axis we represent the space-regularity of initial data, where the value $s$ stands for Gevrey spaces of order $s$ (so that  higher values of $s$ mean lower regularity).  The curve is $s=(1-\alpha)^{-1}$.  

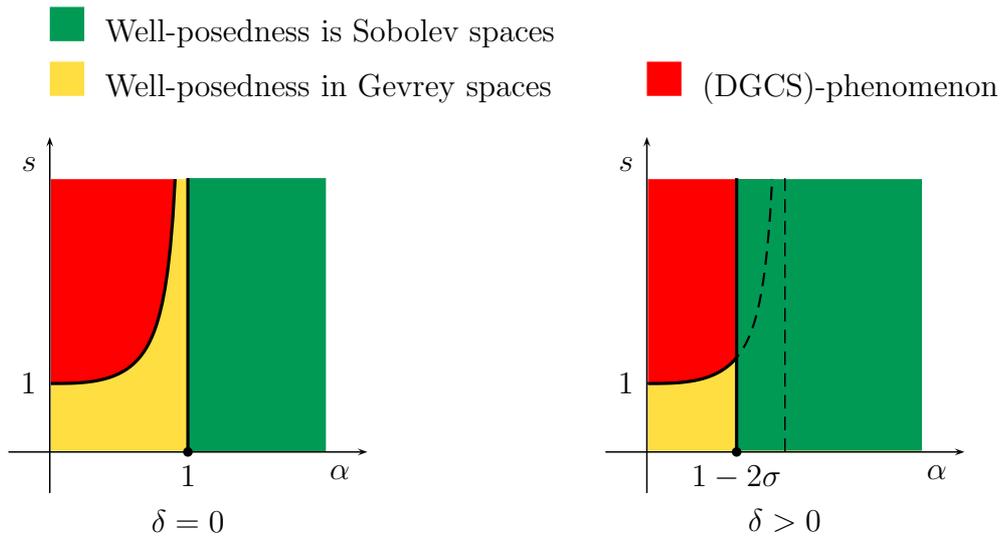
\def\grafico{\psplot{0}{0.95}{0.5 1 x 3 exp sub div}}
\psset{unit=10ex}

\noindent
\hfill
\pspicture(-0.5,-1)(2.5,3.5)
\psclip{\psframe[linestyle=none](0,0)(2,2)}
\pscustom[fillstyle=solid,fillcolor=Goldenrod,linestyle=none]{
\grafico
\psline[linestyle=none](1,2)(1,0)(0,0)}
\pscustom[fillstyle=solid,fillcolor=red,linestyle=none]{
\grafico
\psline[linestyle=none](0,2)}
{\psset{linewidth=1.5\pslinewidth}\grafico}%
\endpsclip
\psframe*[linecolor=ForestGreen](1,0)(2,2)
\psline[linewidth=0.7\pslinewidth]{->}(-0.3,0)(2.3,0)
\psline[linewidth=0.7\pslinewidth]{->}(0,-0.3)(0,2.3)
\psline[linewidth=1.5\pslinewidth](1,0)(1,2)
\uput[-90](2.1,0){$\alpha$}
\uput[-90](1,0){$1$}
\uput[180](0,2.1){$s$}
\uput[180](0,0.5){$1$}
\psdot(1,0)
\rput(1,-0.5){$\delta=0$}
\psframe*[linecolor=Goldenrod](0,2.6)(0.25,2.85)
\rput[Bl](0.4,2.6){Well-posedness in Gevrey spaces}
\psframe*[linecolor=ForestGreen](0,3)(0.25,3.25)
\rput[Bl](0.4,3){Well-posedness is Sobolev spaces}
\endpspicture
\hfill\hfill
\pspicture(-0.5,-1)(2.5,3.5)
\psframe*[linecolor=red](0,2.6)(0.25,2.85)
\rput[Bl](0.4,2.6){(DGCS)-phenomenon}
\psclip{\psframe[linestyle=none](0,0)(2,2)}
\pscustom[fillstyle=solid,fillcolor=Goldenrod,linestyle=none]{
\grafico
\psline[linestyle=none](1,2)(2,2)(2,0)(0,0)}
\pscustom[fillstyle=solid,fillcolor=red,linestyle=none]{
\grafico
\psline[linestyle=none](0,2)}
{\psset{linewidth=1.5\pslinewidth}\grafico}%
\psframe[fillstyle=solid,fillcolor=ForestGreen,linestyle=none](0.65,0)(2,2.1)
{\psset{linestyle=dashed}\grafico}%
\endpsclip
\psclip{\pscustom[linestyle=none]{\grafico\lineto(0,0.5)}}
\endpsclip
\psline[linewidth=0.7\pslinewidth]{->}(-0.3,0)(2.3,0)
\psline[linewidth=0.7\pslinewidth]{->}(0,-0.3)(0,2.3)
\psline[linewidth=0.7\pslinewidth,linestyle=dashed](1,0)(1,2)
\psline[linewidth=1.5\pslinewidth](0.65,0)(0.65,2)
\uput[-90](2.1,0){$\alpha$}
\uput[-90](0.65,0){$1-2\sigma$}
\psdot(0.65,0)
\uput[180](0,2.1){$s$}
\uput[180](0,0.5){$1$}
\rput(1,-0.5){$\delta>0$}
\endpspicture
\hfill\mbox{}

In the non-dissipative case $\delta=0$ we have well-posedness in Sobolev spaces if $\alpha\geq 1$ (this is a classical result), while for $\alpha\in(0,1)$ Theorem~\ref{thmbibl:gevrey} provides global well-posedness in Gevrey spaces of order $s<(1-\alpha)^{-1}$ and local well-posedness if $s=(1-\alpha)^{-1}$, and Theorem~\ref{thmbibl:dgcs} provides the (DGCS)-phenomenon for $s>(1-\alpha)^{-1}$. Finally, Theorem~\ref{thm:gevrey-critical} of this paper shows that in the critical case $s=(1-\alpha)^{-1}$ the solution can lose as many derivatives as possible after a finite time.

In the dissipative case $\delta>0$, the strong damping moves to the left the vertical line that represents the boundary of the region with Sobolev well-posedness. More precisely, Theorem~\ref{thmbibl:sobolev} provides well-posedness in Sobolev spaces if $\alpha>1-2\sigma$, while for $\alpha<1-2\sigma$ we have the same picture as in the non-dissipative case, again provided by Theorems~\ref{thmbibl:hyper},~\ref{thmbibl:gevrey},~\ref{thmbibl:dgcs}, and by Theorem~\ref{thm:gevrey-critical} of this paper. Finally, in the critical case $\alpha=1-2\sigma$, Theorem~\ref{thmbibl:sobolev} provides well-posedness in Sobolev spaces if $\delta$ is large enough, while for $\delta$ small enough Theorem~\ref{thm:damping-critical} shows that the (DGCS)-phenomenon is again possible, with instantaneous loss of as many derivatives as possible.  


\setcounter{equation}{0}
\section{Proofs}\label{sec:proofs}

\subsection{Asymptotic behavior of \holder\ constants}

In this subsection we prove a simple, but somewhat counterintuitive, result. The idea is that, under suitable assumptions, the \holder\ constant of the sum of two functions is not the sum of the \holder\ constants, but the maximum. The result holds true for functions between metric spaces, but we state it just in the setting of propagation speeds. In the sequel,
\begin{equation}
\Hold_{\alpha}(c):=\sup\left\{\frac{|c(t)-c(s)|}{|t-s|^{\alpha}}:(t,s)\in[0,+\infty)^{2},\ t\neq s\right\}
\nonumber
\end{equation}
denotes the \holder\ constant of a propagation speed $c(t)$.

\begin{lemma}[Asymptotic \holder\ constant of a sum]\label{lemma:holder}

Let $f_{n}:[0,+\infty)\to\re$ and $g_{n}:[0,+\infty)\to\re$ be two sequences of functions, and let $\alpha\in(0,1)$ be a real number.

Let us assume that
\begin{itemize}

\item  $f_{n}$ and $g_{n}$ are \holder\ continuous of order $\alpha$ for every $n\in\n$,

\item  $g_{n}\to 0$ uniformly in $[0,+\infty)$,

\item  there exists a real number $L$ such that
\begin{equation}
|f_{n}(t_{1})-f_{n}(t_{2})|\leq L|t_{1}-t_{2}|
\qquad
\forall n\in\n,
\quad
\forall (t_{1},t_{2})\in[0,+\infty)^{2}.
\label{hp:fn-lip}
\end{equation}

\end{itemize}

Then it turns out that
\begin{equation}
\limsup_{n\to +\infty}\Hold_{\alpha}(f_{n}+g_{n})\leq\max\left\{\limsup_{n\to +\infty}\Hold_{\alpha}(f_{n}),\limsup_{n\to +\infty}\Hold_{\alpha}(g_{n})\strut\right\}.
\nonumber
\end{equation}

\end{lemma}

\begin{proof}

For every real number $\ep>0$, let us choose real numbers $\Delta>0$ and $\eta>0$ such that
\begin{equation}
L\Delta^{1-\alpha}\leq\ep,
\qquad\qquad
2\eta\leq\ep\Delta^{\alpha},
\label{defn:delta-eta}
\end{equation}
and let $n_{0}\in\n$ be such that
\begin{equation}
|g_{n}(t)|\leq\eta
\qquad
\forall t\geq 0,
\quad
\forall n\geq n_{0}.
\label{eqn:gn-to-0}
\end{equation}

We claim that
\begin{equation}
|(f_{n}+g_{n})(t_{1})-(f_{n}+g_{n})(t_{2})|\leq\left(\max\left\{\Hold_{\alpha}(f_{n}),\Hold_{\alpha}(g_{n})\strut\right\}+\ep\right)|t_{1}-t_{2}|^{\alpha}
\label{claim:fn+gn}
\end{equation}
for every $n\geq n_{0}$, and every pair $(t_{1},t_{2})$ of nonnegative real numbers. To this end, we distinguish two cases according to the size of $t_{1}-t_{2}$.
\begin{itemize}

\item If $|t_{1}-t_{2}|\leq \Delta$, then from (\ref{hp:fn-lip}) and the first relation in (\ref{defn:delta-eta}) we obtain that
\begin{eqnarray*}
|(f_{n}+g_{n})(t_{1})-(f_{n}+g_{n})(t_{2})| & \leq & |f_{n}(t_{1})-f_{n}(t_{2})|+|g_{n}(t_{1})-g_{n}(t_{2})|
\\[0.5ex]
& \leq &  L|t_{1}-t_{2}|+\Hold_{\alpha}(g_{n})|t_{1}-t_{2}|^{\alpha}
\\[0.5ex]
& = &  \left(L|t_{1}-t_{2}|^{1-\alpha}+\Hold_{\alpha}(g_{n})\right)|t_{1}-t_{2}|^{\alpha}
\\[0.5ex]
& \leq &  \left(\ep+\Hold_{\alpha}(g_{n})\right)|t_{1}-t_{2}|^{\alpha},
\end{eqnarray*}
which implies (\ref{claim:fn+gn}) in this first case.

\item  If $|t_{1}-t_{2}|\geq \Delta$, then from (\ref{eqn:gn-to-0}) and the second relation in (\ref{defn:delta-eta}) we obtain that
\begin{eqnarray*}
|(f_{n}+g_{n})(t_{1})-(f_{n}+g_{n})(t_{2})| & \leq & |f_{n}(t_{1})-f_{n}(t_{2})|+|g_{n}(t_{1})|+|g_{n}(t_{2})|
\\[0.5ex]
& \leq &  \Hold_{\alpha}(f_{n})|t_{1}-t_{2}|^{\alpha}+2\eta
\\[0.5ex]
& \leq &  \left(\Hold_{\alpha}(f_{n})+\ep\right)|t_{1}-t_{2}|^{\alpha},
\end{eqnarray*}
which implies (\ref{claim:fn+gn}) also in this second case.

\end{itemize}

Since $\ep$ is arbitrary, the conclusion follows from (\ref{claim:fn+gn}).
\end{proof}


\subsection{The basic ingredient}\label{sec:basic}

This subsection is the technical core of the paper. We show that every given smooth propagation speed can be slightly modified, with a negligible effect on its upper/lower bounds and on its \holder\ constant, in order to produce a resonance effect with a large enough frequency~$\lambda$.

Let $c_{0}:[0,+\infty)\to(0,+\infty)$ be a positive function of class $C^{2}$. Let $\delta\geq 0$ and $\sigma\in(0,1/2)$ be two real numbers. For every $(\ep,\lambda,t)\in(0,+\infty)\times(0,+\infty)\times[0,+\infty)$, let us consider the functions
\begin{eqnarray}
a(\lambda,t) & := & \lambda\int_{0}^{t}c_{0}(s)^{1/2}\,ds, 
\label{defn:a}   \\[1ex]
b(\ep,\lambda,t) & := & \frac{\ep\lambda}{2}\int_{0}^{t}\frac{\sin^{2}(a(\lambda,s))}{c_{0}(s)^{1/2}}\,ds-\frac{1}{4}\log\frac{c_{0}(t)}{c_{0}(0)}-\delta\lambda^{2\sigma}t,
\label{defn:b}    \\[1ex]
\gamma(\ep,\lambda,t) & := & c_{0}(t)-\ep\sin(2a(\lambda,t))-\frac{\ep^{2}}{4}\frac{\sin^{4}(a(\lambda,t))}{c_{0}(t)}-\frac{5}{16}\frac{1}{\lambda^{2}}\left[\frac{c_{0}'(t)}{c_{0}(t)}\right]^{2}
\nonumber  \\[1ex]
& & \mbox{}+\frac{\ep}{2\lambda}\frac{c_{0}'(t)}{c_{0}(t)^{3/2}}\sin^{2}(a(\lambda,t))+\frac{1}{4\lambda^{2}}\frac{c_{0}''(t)}{c_{0}(t)}+\frac{\delta^{2}}{\lambda^{2-4\sigma}}.
\label{defn:gamma}
\end{eqnarray}

We observe that, in the special case where $c_{0}(t)\equiv m^{2}$ is a positive constant, we obtain the same functions that were used in~\cite[section~6]{gg:cjs-strong}. If in addition $m=1$ and $\delta=0$, we obtain the functions that were originally introduced in~\cite{dgcs}.

With a long but elementary computation, one can check that the function
\begin{equation}
w(\ep,\lambda,t):=\sin(a(\lambda,t))\exp(b(\ep,\lambda,t))
\label{defn:w}
\end{equation}
satisfies
\begin{equation}
\frac{\partial^{2}w}{\partial t^{2}}(\ep,\lambda,t)+2\delta\lambda^{2\sigma}\frac{\partial w}{\partial t}(\ep,\lambda,t)+\lambda^{2}\gamma(\ep,\lambda,t)w(\ep,\lambda,t)=0
\nonumber
\end{equation}
for every admissible value of the variables.

Admittedly, at a first glance it might be not so intuitive why this should be true and therefore, for the convenience of the reader, in appendix~\ref{appendix} we show a heuristic argument that leads to these definitions.

Our goal is showing that, for suitable values of $\ep$ and $\lambda$, the energy of the solution $w$ grows exponentially with time. To this end, we start by estimating from below the growth of $b(\ep,\lambda,t)$.

\begin{lemma}

Let $c_{0}:[0,+\infty)\to(0,+\infty)$ be a function that satisfies the strict hyperbolicity assumption (\ref{hp:c-sh}). Let us assume in addition that $c_{0}$ is of class $C^{1}$, and there exists a constant $L_{0}$ such that $|c_{0}'(t)|\leq L_{0}$ for every $t\geq 0$. Let $b(\ep,\lambda,t)$ be the function defined in (\ref{defn:b}).

Then it turns out that
\begin{equation}
b(\ep,\lambda,t)\geq\frac{\ep\lambda}{4\mu_{2}^{1/2}}\left(1-\frac{L_{0}}{4\mu_{1}^{3/2}}\frac{1}{\lambda}\right)t-\delta\lambda^{2\sigma}t-\frac{\ep}{8(\mu_{1}\mu_{2})^{1/2}}-\frac{1}{4}\log\frac{\mu_{2}}{\mu_{1}}
\label{th:est-b}
\end{equation}
for every $(\ep,\lambda,t)\in(0,+\infty)\times(0,+\infty)\times[0,+\infty)$.

\end{lemma}

\begin{proof}

From the strict hyperbolicity assumption (\ref{hp:c-sh}) we deduce that
\begin{equation}
b(\ep,\lambda,t)\geq\frac{\ep\lambda}{2\mu_{2}^{1/2}}\int_{0}^{t}\sin^{2}(a(\lambda,s))\,ds-\frac{1}{4}\log\frac{\mu_{2}}{\mu_{1}}-\delta\lambda^{2\sigma}t.
\nonumber
\end{equation}

Moreover, by elementary trigonometry we know that
\begin{equation}
\int_{0}^{t}\sin^{2}(a(\lambda,s))\,ds=\frac{t}{2}-\frac{1}{2}\int_{0}^{t}\cos(2a(\lambda,s))\,ds.
\nonumber
\end{equation}

Therefore, it remains to show that
\begin{equation}
\left|\int_{0}^{t}\cos(2a(\lambda,s))\,ds\right|\leq\frac{1}{2\mu_{1}^{1/2}}\cdot\frac{1}{\lambda}+\frac{L_{0}}{4\mu_{1}^{3/2}}\cdot\frac{t}{\lambda}.
\label{est:int-cos}
\end{equation}

In order to estimate this oscillating integral, we integrate by parts in the usual way, and we obtain that (here primes denote derivatives with respect to the variable $s$)
\begin{eqnarray}
\int_{0}^{t}\cos(2a(\lambda,s))\,ds & = & \int_{0}^{t}2a'(\lambda,s)\cos(2a(\lambda,s))\cdot\frac{1}{2a'(\lambda,s)}\,ds
\nonumber
\\[1ex]
& = & \left[\frac{\sin(2a(\lambda,s))}{2a'(\lambda,s)}\right]_{s=0}^{s=t}+\frac{1}{2}\int_{0}^{t}\sin(2a(\lambda,s))\cdot\frac{a''(\lambda,s)}{a'(\lambda,s)^{2}}\,ds.
\nonumber
\end{eqnarray}

Now we recall that $a(\lambda,s)$ is defined by (\ref{defn:a}), and therefore
\begin{equation}
\int_{0}^{t}\cos(2a(\lambda,s))\,ds=\frac{\sin(2a(\lambda,t))}{2\lambda c_{0}(t)^{1/2}}+\frac{1}{4}\int_{0}^{t}\sin(2a(\lambda,s))\cdot\frac{c_{0}'(s)}{\lambda c_{0}(s)^{3/2}}\,ds.
\nonumber
\end{equation}

Exploiting again the strict hyperbolicity (\ref{hp:c-sh}), and the uniform bound on $c_{0}'(t)$, we obtain (\ref{est:int-cos}). 
\end{proof}


Let $\{\lambda_{n}\}$ and $\{\ep_{n}\}$ be two sequences of positive real numbers such that $\lambda_{n}\to +\infty$ and $\ep_{n}\to 0$ as $n\to +\infty$. For every positive integer $n$, let us define
\begin{equation}
c_{n}(t):=\gamma(\ep_{n},\lambda_{n},t)
\qquad
\forall t\geq 0.
\label{defn:cn}
\end{equation}

Let $w_{n}(t)$ denote the solution to the problem
\begin{equation}
w_{n}''(t)+2\delta\lambda_{n}^{2\sigma}w_{n}'(t)+\lambda_{n}^{2}c_{n}(t)w_{n}(t)=0,
\nonumber
\end{equation}
with initial data
\begin{equation}
w_{n}(0)=0,
\qquad
w_{n}'(0)=1.
\nonumber
\end{equation}

The key properties of $c_{n}(t)$ and $w_{n}(t)$ are stated in the following result.

\begin{prop}[Activation of large enough frequencies]\label{prop:activator}

Let $\delta\geq 0$ and $\sigma\in(0,1/2)$ be two real numbers. Let $c_{0}:[0,+\infty)\to(0,+\infty)$ be a function that satisfies the strict hyperbolicity assumption (\ref{hp:c-sh}). Let us assume in addition that $c_{0}$ is of class $C^{3}$ with
\begin{equation}
\sup\left\{|c'(t)|+|c''(t)|+|c'''(t)|:t\geq 0\strut\right\}<+\infty.
\label{hp:C3}
\end{equation}

Let $\{\lambda_{n}\}$ and $\{\ep_{n}\}$ be two sequences of positive real numbers such that $\lambda_{n}\to +\infty$ and $\ep_{n}\to 0$ as $n\to +\infty$, and 
\begin{equation}
\limsup_{n\to+\infty}(\ep_{n}\lambda_{n}^{\alpha})<+\infty.
\label{hp:limsup-eln}
\end{equation}

Let us define $c_{n}(t)$ and $w_{n}(t)$ as above, and let us set
\begin{equation}
\mu_{3}:=\frac{\mu_{1}\min\{1,\mu_{1}\}}{2\mu_{2}^{2}},
\qquad\qquad
\mu_{4}:=\frac{1}{4\mu_{2}^{1/2}}.
\label{defn:mu34}
\end{equation}

Then the following statements hold true.
\begin{itemize}

\item  \emph{(Uniform convergence)} It turns out that
\begin{equation}
c_{n}\to c_{0}
\quad
\mbox{uniformly in }[0,+\infty),
\label{th:cn-to-c0}
\end{equation}

\item  \emph{(Asymptotic behavior of \holder\ constant)} It turns out that
\begin{equation}
\limsup_{n\to +\infty}\Hold_{\alpha}(c_{n})\leq\max\left\{\Hold_{\alpha}(c_{0}),2\mu_{2}^{\alpha/2}\limsup_{n\to +\infty}\left(\ep_{n}\lambda_{n}^{\alpha}\right)\right\}.
\label{th:cn-holder}
\end{equation}

\item  \emph{(Exponential growth of the solution)} For every $n$ large enough it turns out that
\begin{equation}
|w_{n}'(t)|^{2}+\lambda_{n}^{2}|w_{n}(t)|^{2}\geq\mu_{3}\exp\left(\mu_{4}\ep_{n}\lambda_{n}\,t-2\delta\lambda_{n}^{2\sigma}t\right)
\qquad
\forall t\geq 0.
\label{th:est-wn}
\end{equation}

\end{itemize}

\end{prop}

\begin{proof}

Since $\lambda_{n}\to +\infty$ and $\ep_{n}\to 0$ as $n\to +\infty$, and since $2-4\sigma>0$, the uniform convergence follows from the strict hyperbolicity (\ref{hp:c-sh}), and from the uniform bounds on $c_{0}(t)$ in the $C^{2}$ norm.

In the sequel we set for simplicity
\begin{equation}
a_{n}(t):=a(\lambda_{n},t),
\qquad\qquad
b_{n}(t):=b(\ep_{n},\lambda_{n},t),
\nonumber
\end{equation}

\subparagraph{\textmd{\textit{Asymptotic behavior of \holder\ constants}}}

The idea is to apply Lemma~\ref{lemma:holder} with
\begin{equation}
f_{n}(t):=c_{0}(t)-\frac{5}{16}\frac{1}{\lambda_{n}^{2}}\left[\frac{c_{0}'(t)}{c_{0}(t)}\right]^{2}+\frac{1}{4\lambda_{n}^{2}}\frac{c_{0}''(t)}{c_{0}(t)},
\nonumber
\end{equation}
and
\begin{equation}
g_{n}(t):=-\ep_{n}\sin(2a_{n}(t))-\frac{\ep_{n}^{2}}{4}\frac{\sin^{4}(a_{n}(t))}{c_{0}(t)}+\frac{\ep_{n}}{2\lambda_{n}}\frac{c_{0}'(t)}{c_{0}(t)^{3/2}}\sin^{2}(a_{n}(t)).
\nonumber
\end{equation}

To begin with, we observe that a function that is bounded and Lipschitz continuous is also \holder\ continuous. Due to the strict hyperbolicity (\ref{hp:c-sh}), and to the bound (\ref{hp:C3}) on the derivatives of $c_{0}$ up to order three, this implies that the four functions
\begin{equation}
\left[\frac{c_{0}'(t)}{c_{0}(t)}\right]^{2},
\qquad\quad
\frac{c_{0}''(t)}{c_{0}(t)},
\qquad\quad
\frac{1}{c_{0}(t)},
\qquad\quad
\frac{c_{0}'(t)}{c_{0}(t)^{3/2}}
\nonumber
\end{equation}
are both Lipschitz continuous and \holder\ continuous of order $\alpha$, and their Lipschitz and \holder\ constants can be estimated in terms of $\alpha$, $\mu_{1}$, $\mu_{2}$, and the supremum in (\ref{hp:C3}). Recalling that $\lambda_{n}\to 0$ as $n\to +\infty$, this is enough to conclude that the sequence $\{f_{n}\}$ satisfies the equi-Lipschitz assumption (\ref{hp:fn-lip}) of Lemma~\ref{lemma:holder}, and
\begin{equation}
\limsup_{n\to +\infty}\Hold_{\alpha}(f_{n})=\Hold(c_{0}).
\label{th:limsup-fn}
\end{equation}

Now let $g_{1,n}(t)$, $g_{2,n}(t)$, $g_{3,n}(t)$ denote the three terms in the definition of $g_{n}(t)$. In order to estimate $g_{1,n}(t)$, from (\ref{defn:a}) we deduce that
\begin{equation}
|a(\lambda,t_{1})-a(\lambda,t_{2})|\leq\lambda\mu_{2}^{1/2}|t_{1}-t_{2}|
\qquad
\forall(t_{1},t_{2})\in[0,+\infty)^{2}.
\label{est:a-lip}
\end{equation}

Then we observe that
\begin{equation}
|\sin(2y)-\sin(2x)|\leq 2|y-x|^{\alpha}
\qquad
\forall(x,y)\in\re^{2}.
\label{est:sin-holder}
\end{equation}

Indeed, if $|y-x|\leq 1$ this inequality follows from the Lipschitz continuity (with constant~2) of the function $\sin(2x)$, because
\begin{equation}
|\sin(2y)-\sin(2x)|\leq 2|y-x|\leq 2|y-x|^{\alpha},
\nonumber
\end{equation}
while if $|y-x|\geq 1$ the same inequality follows from the boundeness, because
\begin{equation}
|\sin(2y)-\sin(2x)|\leq 2\leq 2|y-x|^{\alpha}.
\nonumber
\end{equation}

From (\ref{est:a-lip}) and (\ref{est:sin-holder}) it follows that
\begin{equation}
|\sin(2a_{n}(t_{1}))-\sin(2a_{n}(t_{2}))|\leq 2|a_{n}(t_{1})-a_{n}(t_{2})|^{\alpha} \leq 2\mu_{2}^{\alpha/2}\lambda_{n}^{\alpha}|t_{1}-t_{2}|^{\alpha},
\nonumber
\end{equation}
which implies that
\begin{equation}
\Hold_{\alpha}(g_{1,n})\leq 2\mu_{2}^{\alpha/2}\ep_{n}\lambda_{n}^{\alpha}
\qquad
\forall n\in\n.
\nonumber
\end{equation}

An analogous argument shows that
\begin{equation}
|\sin^{4}(a_{n}(t_{1}))-\sin^{4}(a_{n}(t_{2}))|\leq H_{1}\mu_{2}^{\alpha/2}\lambda_{n}^{\alpha}|t_{1}-t_{2}|^{\alpha},
\nonumber
\end{equation}
where $H_{1}$ is the \holder\ constant of the function $\sin^{4}x$. Now we recall that the product of two functions that are bounded and \holder\ continuous of order $\alpha$ is again  \holder\ continuous of order $\alpha$, with a constant that depends on the two bounds and on the two \holder\ constants. It follows that
\begin{equation}
\Hold_{\alpha}(g_{2,n})\leq H_{2}\ep_{n}^{2}(\lambda_{n}^{\alpha}+1)
\qquad
\forall n\in\n,
\label{est:g2-hold}
\end{equation}
where $H_{2}$ depends only on $H_{1}$, $\alpha$, $\mu_{1}$, $\mu_{2}$, and the supremum in (\ref{hp:C3}).  In an analogous way, we deduce also that
\begin{equation}
\Hold_{\alpha}(g_{3,n})\leq H_{3}\frac{\ep_{n}}{\lambda_{n}}(\lambda_{n}^{\alpha}+1)
\qquad
\forall n\in\n
\label{est:g3-hold}
\end{equation}
for a suitable constant $H_{3}$.

Due to (\ref{hp:limsup-eln}), the right-hand sides of (\ref{est:g2-hold}) and (\ref{est:g3-hold}) tend to 0 as $n\to +\infty$, and therefore
\begin{equation}
\limsup_{n\to +\infty}\Hold_{\alpha}(g_{n})=\limsup_{n\to +\infty}\Hold_{\alpha}(g_{1,n})\leq 2\mu_{2}^{\alpha/2}\limsup_{n\to +\infty}\left(\ep_{n}\lambda_{n}^{\alpha}\right).
\label{th:limsup-gn}
\end{equation}

At this point, the conclusion follows from (\ref{th:limsup-fn}), (\ref{th:limsup-gn}), and Lemma~\ref{lemma:holder}.

\subparagraph{\textmd{\textit{Exponential growth of $w_{n}(t)$}}}

To begin with, we observe that we have an explicit formula for $w_{n}(t)$, namely
\begin{equation}
w_{n}(t)=\frac{1}{\lambda_{n}c_{0}(0)^{1/2}}w(\ep_{n},\lambda_{n},t)
\qquad
\forall t\geq 0,
\nonumber
\end{equation}
where $w(\ep,\lambda,t)$ is the function defined in (\ref{defn:w}). This implies that
\begin{equation}
w_{n}'(t)=\frac{1}{\lambda_{n}c_{0}(0)^{1/2}}\left\{a_{n}'(t)\cos(a_{n}(t))+b_{n}'(t)\sin(a_{n}(t))\right\}\exp(b_{n}(t)),
\nonumber
\end{equation}
and therefore
\begin{equation}
|w_{n}'(t)|^{2}+\lambda_{n}^{2}|w_{n}(t)|^{2}=R_{n}(t)\exp(2b_{n}(t))
\qquad
\forall t\geq 0,
\label{E-wn=Rn-bn}
\end{equation}
where (for the sake of shortness we do not write explicitly the dependence on $t$ in the terms of the right-hand side)
\begin{equation}
R_{n}(t):=\frac{1}{\lambda_{n}^{2}c_{0}(0)}\left\{(a_{n}')^{2}\cos^{2}a_{n}+(b_{n}')^{2}\sin^{2}a_{n}+\lambda_{n}^{2}\sin^{2}a_{n}+2a_{n}'b_{n}'\cos a_{n}\sin a_{n}\right\}.
\nonumber
\end{equation}

Now we recall that $a_{n}'(t)=\lambda_{n}c_{0}(t)^{1/2}$, and therefore
\begin{eqnarray}
R_{n}(t) & \geq & \frac{c_{0}(t)}{c_{0}(0)}\cos^{2}(a_{n}(t))+\frac{1}{c_{0}(0)}\sin^{2}(a_{n}(t))-\frac{2c_{0}(t)^{1/2}}{c_{0}(0)}\frac{1}{\lambda_{n}}|b_{n}'(t)|
\nonumber
\\[1ex]
& \geq & \frac{\min\left\{1,\mu_{1}\right\}}{\mu_{2}}-\frac{2\mu_{2}^{1/2}}{\mu_{1}}\frac{1}{\lambda_{n}}|b_{n}'(t)|.
\nonumber
\end{eqnarray}

Now we recall that
\begin{equation}
b_{n}'(t)=\lambda_{n}\left\{\frac{\ep_{n}}{2}\frac{\sin^{2}(a_{n}(t))}{c_{0}(t)^{1/2}}-\frac{1}{4\lambda_{n}}\frac{c_{0}'(t)}{c_{0}(t)}-\frac{\delta}{\lambda_{n}^{1-2\sigma}}\right\}.
\nonumber
\end{equation}

Since $\lambda_{n}\to +\infty$ and $\ep_{n}\to 0$ as $n\to +\infty$, and since $1-2\sigma>0$, from the strict hyperbolicity condition (\ref{hp:c-sh}) and the bounds on $c_{0}(t)$ in the $C^{1}$ norm, it follows that $|b_{n}'(t)|/\lambda_{n}\to 0$ uniformly in $[0,+\infty)$, and therefore
\begin{equation}
R_{n}(t)\geq\frac{\min\left\{1,\mu_{1}\right\}}{2\mu_{2}}
\qquad
\forall t\geq 0,
\label{est:Rn}
\end{equation}
provided that $n$ is sufficiently large.

Moreover, from (\ref{th:est-b}) it follows that
\begin{equation}
2b_{n}(t)\geq\frac{\ep_{n}\lambda_{n}}{4\mu_{2}^{1/2}}t-2\delta\lambda_{n}^{2\sigma}t-\log\frac{\mu_{2}}{\mu_{1}}
\qquad
\forall t\geq 0,
\label{est:bn}
\end{equation}
again when $n$ is sufficiently large (because we absorbed the third term in the right-hand side of (\ref{th:est-b}) by changing the coefficient of the logarithm).

Plugging (\ref{est:Rn}) and (\ref{est:bn}) into (\ref{E-wn=Rn-bn}), we obtain (\ref{th:est-wn}).
\end{proof}


\subsection{Universal activators}\label{sec:baire}

In this subsection the Baire category theorem discloses all its power. In Proposition~\ref{prop:activator} we used the basic ingredients in order to cook up a propagation speed that activates a large enough frequency. As far as we know, that propagation speed might activate just that special frequency. Now we produce, or better we let the Baire category theorem produce, a residual set of propagation speeds that activate infinitely many frequencies in the same time. We call them universal activators. The formal definition is the following. 

\begin{defn}[Universal activators]\label{defn:activator}
\begin{em}

Let $\{\lambda_{i}\}$ be a sequence of positive real numbers, and let $\mu_{3}$ and $\mu_{5}$ be two positive real numbers. 

A continuous function $c:[0,+\infty)\to\re$ is called a \emph{universal activator} for the sequence $\{\lambda_{n}\}$ with order $(\mu_{3},\mu_{5})$ if the sequence $\{w_{i}(t)\}$ of solutions to equation
\begin{equation}
w_{i}''(t)+2\delta\lambda_{i}^{2\sigma}w_{i}'(t)+\lambda_{i}^{2}c(t)w_{i}(t)=0
\label{activator-eqn}
\end{equation}
with initial data
\begin{equation}
w_{i}(0)=0,
\qquad
w_{i}'(0)=1
\label{activator-data}
\end{equation}
satisfies the exponential growth condition
\begin{equation}
\limsup_{i\to +\infty}\left(|w_{i}'(t)|^{2}+\lambda_{i}^{2}|w_{i}(t)|^{2}\right)\exp\left(-\mu_{5}\lambda_{i}^{1-\alpha}t+2\delta\lambda_{i}^{2\sigma}t\right)\geq\mu_{3}
\qquad
\forall t\geq 0.
\label{th:activator}
\end{equation}

\end{em}
\end{defn}

The following result is the key tool in the proof of our main results, but it could also provide alternative and shorter proofs of Theorem~\ref{thmbibl:dgcs}.

\begin{thm}[Existence of universal activators]\label{thm:activators}

Let $\mu_{1}$, $\mu_{2}$, $\alpha$, $H$ be real numbers satisfying (\ref{hp:accH}). Let  us define $\mu_{3}$ and $\mu_{4}$ as in (\ref{defn:mu34}), and let us set
\begin{equation}
\mu_{5}:=\frac{H\mu_{4}}{4\mu_{2}^{\alpha/2}}=\frac{H}{16\mu_{2}^{(1+\alpha)/2}}.
\nonumber
\end{equation}

Let $\{\lambda_{i}\}$ be a sequence of positive real numbers such that $\lambda_{i}\to +\infty$ as $i\to +\infty$.

Then the set of all propagation speeds $c\in\PS(\mu_{1},\mu_{2},\alpha,H)$ that are universal activators for the sequence $\{\lambda_{n}\}$ with order $(\mu_{3},\mu_{5})$ is residual.

\end{thm}

\begin{proof}

Let us consider the set $\CC$ of ``non universal activators'', namely the set of all propagation speeds $c\in\PS(\mu_{1},\mu_{2},\alpha,H)$ for which (\ref{th:activator}) is false, and therefore
\begin{equation}
\exists t\geq 0
\qquad
\limsup_{i\to +\infty}\left(|w_{i}'(t)|^{2}+\lambda_{i}^{2}|w_{i}(t)|^{2}\right)\exp\left(-\mu_{5}\lambda_{i}^{1-\alpha}t+2\delta\lambda_{i}^{2\sigma}t\right)<\mu_{3},
\nonumber
\end{equation}
or equivalently
\begin{multline}
\exists t\geq 0
\quad
\exists\eta>0
\quad
\exists j\in\n
\quad
\forall i\geq j
\\[0.5ex]
|w_{i}'(t)|^{2}+\lambda_{i}^{2}|w_{i}(t)|^{2}\leq(\mu_{3}-\eta)\exp\left(\mu_{5}\lambda_{i}^{1-\alpha}t-2\delta\lambda_{i}^{2\sigma}t\right).
\label{defn:non-activ}
\end{multline}

\paragraph{\textmd{\textit{Quantitative non-activators}}}

In order to make the previous statement more quantitative, we introduce the set $\CC_{k}$ of all admissible propagation speeds such that
\begin{equation}
\exists t\in\left[0,k\right]
\quad
\forall i\geq k
\qquad
|w_{i}'(t)|^{2}+\lambda_{i}^{2}|w_{i}(t)|^{2}\leq\left(\mu_{3}-\frac{1}{k}\right)\exp\left(\mu_{5}\lambda_{i}^{1-\alpha}t-2\delta\lambda_{i}^{2\sigma}t\right).
\nonumber
\end{equation}

In words, now $t$ is confined in the compact interval $[0,k]$, and we have chosen $\eta=1/k$ and $j=k$ in (\ref{defn:non-activ}). It turns out that the set $\CC$ is the union of all $\CC_{k}$'s. Indeed, if some propagation speed satisfies (\ref{defn:non-activ}), then the same propagation speed belongs to $\CC_{k}$ provided that $k$ satisfies
\begin{equation}
k\geq t,
\qquad\qquad
\frac{1}{k}\leq\eta,
\qquad\qquad
k\geq j.
\nonumber
\end{equation}

The proof is complete if we show that $\CC_{k}$ is a closed set with empty interior for every positive integer $k$.

\paragraph{\textmd{\textit{The set $\CC_{k}$ is closed}}}

Let $k$ be a \emph{fixed} positive integer. Let $\{c_{n}\}\subseteq\CC_{k}$ be any sequence, and let us assume that $c_{n}(t)\to c_{\infty}(t)$ uniformly in $[0,+\infty)$. We claim that $c_{\infty}\in\CC_{k}$.

For every positive integer $i$, let $w_{i,n}(t)$ denote the solution to (\ref{activator-eqn})--(\ref{activator-data}) with $c:=c_{n}$, and let $w_{i,\infty}$ denote the solution with  with $c:=c_{\infty}$. From the definition of $\CC_{k}$ we know that, for every $n\in\n$, there exists $t_{n}\in[0,k]$ such that
\begin{equation}
|w_{i,n}'(t_{n})|^{2}+\lambda_{i}^{2}|w_{i,n}(t_{n})|^{2}\leq\left(\mu_{3}-\frac{1}{k}\right)\exp\left(\mu_{5}\lambda_{i}^{1-\alpha}t_{n}-2\delta\lambda_{i}^{2\sigma}t_{n}\right)
\qquad
\forall i\geq k.
\label{est:win-tin}
\end{equation}

Up to subsequences (not relabeled), we can always assume that $t_{n}\to t_{\infty}\in[0,k]$. Moreover, we know that solutions to (\ref{activator-eqn})--(\ref{activator-data}) depend is a continuous way on the propagation speed, in the sense that
\begin{equation}
c_{n}\to c_{\infty}
\quad
\mbox{uniformly in }[0,T]
\quad\Longrightarrow\quad
w_{i,n}\to w_{i,\infty}
\quad
\mbox{in }C^{1}([0,T])
\nonumber
\end{equation}
for every $T>0$. This implies in particular that
\begin{equation}
|w_{i,n}'(t)|^{2}+\lambda_{i}^{2}|w_{i,n}(t)|^{2}\to |w_{i,\infty}'(t)|^{2}+\lambda_{i}^{2}|w_{i,\infty}(t)|^{2}
\qquad
\mbox{uniformly in }[0,k],
\nonumber
\end{equation}
which in turn implies that we can pass to the limit in (\ref{est:win-tin}) and deduce that
\begin{equation}
|w_{i,\infty}'(t_{\infty})|^{2}+\lambda_{i}^{2}|w_{i,\infty}(t_{\infty})|^{2}\leq\left(\mu_{3}-\frac{1}{k}\right)\exp\left(\mu_{5}\lambda_{i}^{1-\alpha}t_{\infty}-2\delta\lambda_{i}^{2\sigma}t_{\infty}\right)
\qquad
\forall i\geq k,
\nonumber
\end{equation}
which proves that $c_{\infty}\in\CC_{k}$.

\paragraph{\textmd{\textit{The set $\CC_{k}$ has empty interior}}}

Let us assume by contradiction that there exist an integer $k_{0}\geq 1$, an admissible propagation speed $c_{0}$, and a real number $\ep_{0}>0$ such that $B_{\PS}(c_{0},\ep_{0})\subseteq\CC_{k_{0}}$, where $B_{\PS}(c_{0},\ep_{0})$ denotes the open ball in $\PS(\mu_{1},\mu_{2},\alpha,H)$ with center in $c_{0}$ and radius $\ep_{0}$. 

\subparagraph{\textmd{\textit{Regularization of the center}}}

Up to a small modification of $c_{0}$, and a small reduction of the radius $\ep_{0}$, we can assume that $c_{0}$ has the following further properties.
\begin{itemize}

\item  It is of class $C^{3}$ and satisfies (\ref{hp:C3}).

\item  It does not saturate the inequalities in (\ref{hp:c-sh}) and (\ref{hp:c-holder}), namely there exists $\ep_{1}>0$ such that
\begin{equation}
\mu_{1}+\ep_{1}\leq c_{0}(t)\leq \mu_{2}-\ep_{1}
\qquad
\forall t\geq 0,
\label{sat:cn}
\end{equation}
and
\begin{equation}
|c_{0}(t)-c_{0}(s)|\leq(1-\ep_{1})H|t-s|^{\alpha}
\qquad
\forall(t,s)\in[0,+\infty)^{2}.
\label{sat:cn-holder}
\end{equation}

\end{itemize}

\subparagraph{\textmd{\textit{Use of rescaled basic ingredient}}}

For every positive integer $n$, let us set
\begin{equation}
\ep_{n}:=\frac{H}{4\mu_{2}^{\alpha/2}}\cdot\frac{1}{\lambda_{n}^{\alpha}},
\label{defn:epn}
\end{equation}
and let us consider the corresponding sequence of propagation speeds $c_{n}(t)$ defined as in (\ref{defn:cn}) starting from the modified version of $c_{0}$. We claim that, for $n$ sufficiently large, $c_{n}\in B_{\PS}(c_{0},\ep_{0})$ but $c_{n}\not\in\CC_{k_{0}}$. This would give a contradiction.

\subparagraph{\textmd{\textit{Final contradiction: $c_{n}\in B_{\PS}(c_{0},\ep_{0})$ for $n$ large enough}}}

Since $\lambda_{n}\to +\infty$ and $\ep_{n}\to 0$, from (\ref{th:cn-to-c0}) it follows that $c_{n}\to c_{0}$ uniformly in $[0,+\infty)$. Due to (\ref{sat:cn}), this implies that $c_{n}(t)$ satisfies the strict hyperbolicity condition (\ref{hp:c-sh}) when $n$ is large enough. Finally, from (\ref{th:cn-holder}), (\ref{sat:cn-holder}) and (\ref{defn:epn}) it follows that
\begin{equation}
\limsup_{n\to +\infty}\Hold_{\alpha}(c_{n})\leq\max\left\{\Hold_{\alpha}(c_{0}),\frac{H}{2}\right\}\leq\max\left\{(1-\ep_{1})H,\frac{H}{2}\right\}.
\nonumber
\end{equation}

Since the latter is strictly less than $H$, the propagation speed $c_{n}$ satisfies (\ref{hp:c-holder}) for $n$ large enough.

In conclusion, we have proved that $c_{n}\in\PS(\mu_{1},\mu_{2},\alpha,H)$ when $n$ is sufficiently large, and $c_{n}\to c_{0}$ uniformly in $[0,+\infty)$, and this implies that $c_{n}\in B_{\PS}(c_{0},\ep_{0})$ when $n$ is large enough.

\subparagraph{\textmd{\textit{Final contradiction: $c_{n}\not\in C_{k_{0}}$ for $n$ large enough}}}

Let us consider the solution $w_{n,n}$ to (\ref{activator-eqn})--(\ref{activator-data}) with $i:=n$ and $c:=c_{n}$. From (\ref{th:est-wn}) we know that, for every $n$ large enough, it turns out that
$$|w_{n,n}'(t)|^{2}+\lambda_{n}^{2}|w_{n,n}(t)|^{2}\geq\mu_{3}\exp\left(\mu_{4}\ep_{n}\lambda_{n}t-2\delta\lambda_{n}^{2\sigma}t\right)=\mu_{3}\exp\left(\mu_{5}\lambda_{n}^{1-\alpha}t-2\delta\lambda_{n}^{2\sigma}t\right)$$
for every $t\geq 0$, which contradicts the definition of $C_{k_{0}}$ as soon as $n\geq k_{0}$.
\end{proof}


\subsection{Proof of Theorem~\ref{thm:gevrey-critical} and Theorem~\ref{thm:damping-critical}}

Let $\{\lambda_{i}\}$ be the sequence of (the square roots of) the eigenvalues of $A$, and let $\{e_{i}\}$ be the corresponding orthonormal basis of $\HH$. Up to extracting a subsequence, we can also assume that
\begin{equation}
\sum_{i=0}^{\infty}\frac{1}{\lambda_{i}^{2}}<+\infty.
\label{hp:li-series}
\end{equation}

Let $c\in\PS(\mu_{1},\mu_{2},\alpha,H)$ be a universal activator for the sequence $\{\lambda_{i}\}$ with some order $(\mu_{3},\mu_{5})$ according to Definition~\ref{defn:activator}. Given a sequence $\{a_{i}\}$ of real numbers, we consider the solution to (\ref{eqn:main}) with initial data
\begin{equation}
u(0)=0,
\qquad
u'(0)=\sum_{i=0}^{\infty}a_{i}e_{i}.
\nonumber
\end{equation}

The solution is
\begin{equation}
u(t)=\sum_{i=0}^{\infty}a_{i}w_{i}(t)e_{i},
\nonumber
\end{equation}
where $\left\{w_{i}(t)\right\}$ is the sequence of solutions to (\ref{activator-eqn})--(\ref{activator-data}). We observe that
\begin{equation}
(u(0),u'(0))\in\GG_{s,r_{0},1/2}(A)\times\GG_{s,r_{0},0}(A)
\quad\Longleftrightarrow\quad
\sum_{i=0}^{\infty}a_{i}^{2}\exp\left(2r_{0}\lambda_{i}^{1/s}\right)<+\infty,
\label{equiv:data-gsr}
\end{equation}
and similarly
\begin{equation}
(u(0),u'(0))\in\GG_{s,\log,1/2}(A)\times\GG_{s,\log,0}(A)
\quad\Longleftrightarrow\quad
\sum_{i=0}^{\infty}a_{i}^{2}\exp\left(\frac{2\lambda_{i}^{1/s}}{\log(2+\lambda_{i})}\right)<+\infty,
\label{equiv:data-dsr}
\end{equation}
while $(u(t),u'(t))\not\in\GG_{-S,\log,1/2}(A)\times\GG_{-S,\log,0}(A)$ if and only if
\begin{equation}
\sum_{i=0}^{\infty}a_{i}^{2}\left(|w_{i}'(t)|^{2}+\lambda_{i}^{2}|w_{i}(t)|^{2}\right)\exp\left(-\frac{2\lambda_{i}^{1/S}}{\log(2+\lambda_{i})}\right)=+\infty.
\label{equiv:t>0}
\end{equation}

In the case of Theorem~\ref{thm:gevrey-critical} we choose
\begin{equation}
a_{i}:=\frac{1}{\lambda_{i}}\exp\left(-r_{0}\lambda_{i}^{1/s}\right)
\qquad
\forall i\in I.
\nonumber
\end{equation}

In this case the convergence of the series in (\ref{equiv:data-gsr}) reduces to (\ref{hp:li-series}). As for (\ref{equiv:t>0}), we observe that the general term of the series can be rewritten as $\Phi_{i}(t)\Psi_{i}(t)$, where
\begin{equation}
\Phi_{i}(t):=\left(|w_{i}'(t)|^{2}+\lambda_{i}^{2}|w_{i}(t)|^{2}\right)\exp\left(-\mu_{5}\lambda_{i}^{1-\alpha}t+2\delta\lambda_{i}^{2\sigma}t\right),
\label{defn:Phi-i}
\end{equation}
and
\begin{equation}
\Psi_{i}(t):=\frac{1}{\lambda_{i}^{2}}\exp\left(-2r_{0}\lambda_{i}^{1/s}+\mu_{5}\lambda_{i}^{1-\alpha}t-2\delta\lambda_{i}^{2\sigma}t-\frac{2\lambda_{i}^{1/S}}{\log(2+\lambda_{i})}\right).
\nonumber
\end{equation}

From the definition of universal activator it turns out that
\begin{equation}
\limsup_{i\to+\infty}\Phi_{i}(t)\geq\mu_{3}>0
\qquad
\forall t>0.
\label{limsup-Phi-i}
\end{equation}

Since $1-\alpha=1/s=1/S>2\sigma$, it turns out that
\begin{equation}
\lim_{i\to+\infty}\Psi_{i}(t)=+\infty
\label{lim-Psi-i}
\end{equation}
for every $t>2r_{0}/\mu_{5}=:t_{0}$. This proves that the series in (\ref{equiv:t>0}) diverges for every $t>t_{0}$.

In the case of Theorem~\ref{thm:damping-critical} we choose
\begin{equation}
a_{i}:=\frac{1}{\lambda_{i}}\exp\left(-\frac{\lambda_{i}^{1/s}}{\log(2+\lambda_{i})}\right).
\nonumber
\end{equation}

As before the convergence of the series in (\ref{equiv:data-dsr}) reduces to (\ref{hp:li-series}), while the general term of the series in (\ref{equiv:t>0}) can be written as $\Phi_{i}(t)\Psi_{i}(t)$, with $\Phi_{i}(t)$ defined by (\ref{defn:Phi-i}), and 
\begin{equation}
\Psi_{i}(t):=\frac{1}{\lambda_{i}^{2}}\exp\left(-\frac{2\lambda_{i}^{1/s}}{\log(2+\lambda_{i})}+\mu_{5}\lambda_{i}^{1-\alpha}t-2\delta\lambda_{i}^{2\sigma}t-\frac{2\lambda_{i}^{1/S}}{\log(2+\lambda_{i})}\right).
\nonumber
\end{equation}

Again the definition of universal activator implies (\ref{limsup-Phi-i}), while the fact that $1-\alpha=1/s=1/S=2\sigma$, and $2\delta<\mu_{5}$, implies (\ref{lim-Psi-i}) for every $t>0$. This proves that the series in (\ref{equiv:t>0}) diverges for every $t>0$ if $2\delta<\mu_{5}$.


\subsection{Proof of Remark~\ref{rmk:data-residual}}\label{sec:rmk}

For the sake of shortness, let $\XX$ denote the complete metric space $\GG_{s,\log,1/2}(A)\times\GG_{s,\log,0}(A)$. Let $\CC$ denote the set of elements of $\XX$ that are initial data of solutions to (\ref{eqn:main}) that do not satisfy (\ref{th:dgcs-gevrey-bis}), so that
\begin{equation}
\exists t>0
\qquad
\|u'(t)\|_{\GG_{-S,\log,0}(A)}^{2}+\|u(t)\|_{\GG_{-S,\log,1/2}(A)}^{2}<+\infty.
\nonumber
\end{equation}

We make this statement more quantitative by introducing the set $\CC_{k}$ of all initial data in $\XX$ that originate solutions to (\ref{eqn:main}) satisfying
\begin{equation}
\exists t\in[1/k,k]
\qquad
\|u'(t)\|_{\GG_{-S,\log,0}(A)}^{2}+\|u(t)\|_{\GG_{-S,\log,1/2}(A)}^{2}\leq k.
\nonumber
\end{equation}

It is possible to show that $\CC$ is the union of all $C_{k}$'s, and that $\CC_{k}$ is a closed subset of $\XX$ for every positive integer $k$ (this requires only that Fourier components of solutions depend continuously on initial data, and that the norm in the spaces $\GG_{-S,\log,\beta}(A)$ is lower semicontinuous with respect to component-wise convergence).

It remains to show that $\CC_{k}$ has empty interior for every positive integer $k$. Let us assume by contradiction that some $\CC_{k_{0}}$ contains the open ball in $\XX$ with center in some $(v_{0},v_{1})\in\XX$ and radius $\ep_{0}>0$. Up to a small reduction of the radius, we can assume that the center $(v_{0},v_{1})$ has only a finite number of Fourier components different from zero, and therefore the corresponding solution $v(t)$ satisfies
\begin{equation}
\|v'(t)\|_{\GG_{-S,\log,0}(A)}^{2}+\|v(t)\|_{\GG_{-S,\log,1/2}(A)}^{2}\leq M_{0}
\qquad
\forall t\in[0,k_{0}]
\nonumber
\end{equation}
for a suitable constant $M_{0}$.

By assumption, we know that equation (\ref{eqn:main}) has a solution $u(t)$, with suitable initial data $(u_{0},u_{1})\in\XX$, that satisfies (\ref{th:dgcs-gevrey-bis}). Due to the linearity of the equation, the solution with initial data $(v_{0}+\ep u_{0},v_{1}+\ep u_{1})$ is $v(t)+\ep u(t)$, and therefore $(v_{0}+\ep u_{0},v_{1}+\ep u_{1})\not\in\CC_{k_{0}}$ for every $\ep\neq 0$.

On the other hand, $(v_{0}+\ep u_{0},v_{1}+\ep u_{1})$ belongs to the ball if $|\ep|$ is small enough, and this provides the required contradiction.


\appendix

\setcounter{equation}{0}
\section{Heuristics for the basic ingredient}\label{appendix}

Let $c_{0}\in\PS(\mu_{1},\mu_{2},\alpha,H)$ be a \emph{given smooth} function. Suppose we want to find a function $c_{\lambda}$, which is close enough to $c_{0}$ in the uniform norm, and has \holder\ constant close to the \holder\ constant of $c_{0}$, such that equation
\begin{equation}
w''(t)+2\delta\lambda^{2\sigma}w'(t)+\lambda^{2}c_{\lambda}(t)w(t)=0
\label{heu:eqn}
\end{equation}
has a solution that grows exponentially with time. To this end, it seems reasonable to look for a solution of the form
\begin{equation}
w(t)=\sin(a(t))\exp(b(t)),
\label{heu:sol}
\end{equation}
for suitable functions $a(t)$ and $b(t)$. Plugging (\ref{heu:sol}) into (\ref{heu:eqn}), with some computations we find that (for the sake of shortness, we do not write explicitly the dependence on $t$)
\begin{equation}
\left\{-(a')^{2}+b''+(b')^{2}+2\delta\lambda^{2\sigma}b'+\lambda^{2}c_{\lambda}\right\}\sin a+\left\{a''+2a'b'+2\delta\lambda^{2\sigma}a'\right\}\cos a=0.
\label{heu:eqn-big}
\end{equation}

In order to simplify the coefficient of $\cos a$, it seems reasonable to consider functions $b(t)$ of the form
\begin{equation}
b(t)=\beta(t)-\delta\lambda^{2\sigma}t,
\nonumber
\end{equation}
so that (\ref{heu:eqn-big}) reduces to
\begin{equation}
\left\{-(a')^{2}+\beta''+(\beta')^{2}-\delta^{2}\lambda^{4\sigma}+\lambda^{2}c_{\lambda}\right\}\sin a+\left\{a''+2a'\beta'\right\}\cos a=0.
\label{heu:eqn-med}
\end{equation}

At this point it would be useful to factor out a $\sin a$ from the coefficient of $\cos a$. Thus we make the ansatz that
\begin{equation}
a''+2a'\beta'=\ep\lambda^{2}\sin^{2}a,
\label{ansatz-b}
\end{equation}
where the square gives us some hope that $\beta'$ could be positive, which means $\beta$ increasing.

Now we recall that $c_{\lambda}(t)$ should hopefully be close to $c_{0}(t)$, and this leads us to a second ansatz that
\begin{equation}
a'(t)=\lambda c_{0}(t)^{1/2}.
\nonumber
\end{equation}

If this is the case, then $a(t)$ is given by (\ref{defn:a}). At this point from (\ref{ansatz-b}) we obtain that
\begin{equation}
\beta'=\frac{\ep\lambda^{2}\sin^{2}a-a''}{2a'},
\nonumber
\end{equation}
from which we compute $\beta(t)$ and therefore also $b(t)$. In this way we obtain (\ref{defn:b}). Finally, from (\ref{heu:eqn-med}) we can compute $c_{\lambda}(t)$, and we obtain exactly (\ref{defn:gamma}). 


\subsubsection*{\centering Acknowledgments}

Both authors are members of the \selectlanguage{italian} ``Gruppo Nazionale per l'Analisi Matematica, la Probabilit\`{a} e le loro Applicazioni'' (GNAMPA) of the ``Istituto Nazionale di Alta Matematica'' (INdAM). 

\selectlanguage{english}



\label{NumeroPagine}

\end{document}